\documentclass[12pt,draftcls,onecolumn,letter]{ieeeconf}

\IEEEoverridecommandlockouts 

\usepackage{graphicx}
\usepackage{color}
\usepackage{subfigure}
\usepackage{algorithm,algpseudocode}

\usepackage{amsmath,amssymb}
\usepackage[noadjust]{cite}

\newcommand{\bbm}{\begin{bmatrix}}
\newcommand{\ebm}{\end{bmatrix}}

\newtheorem{corollary}{\bfseries Corollary}

\newtheorem{lemma}{\bfseries Lemma}

\newtheorem{remark}{\bfseries Remark}
\newtheorem{theorem}{\bfseries Theorem}

\definecolor{MG}{rgb}{0,0.45,0.08}

\graphicspath{{img/}}

\title{\LARGE \bf
Consensus under Misaligned Orientations}

\author{
Hyo-Sung Ahn${}^{\dag}$, Minh Hoang Trinh${}^{\dag}$ \& Byung-Hun Lee${}^{\dag}$
\thanks{${}^{\dag}$School of Mechanical Engineering, Gwangju Institute of Science and Technology (GIST), 123 Cheomdan-gwagiro, Buk-gu, Gwangju, Republic of Korea. E-mails: {\tt\small hyosung@gist.ac.kr,trinhhoangminh@gist.ac.kr, bhlee@gist.ac.kr } }
}

\begin{document}
\maketitle
\thispagestyle{empty}
\pagestyle{empty}

\begin{abstract} 
This paper presents a consensus algorithm under misaligned orientations, which is defined as (i) misalignment to global coordinate frame of local coordinate frames, (ii) biases in control direction or sensing direction, or (iii) misaligned virtual global coordinate frames. After providing a mathematical formulation, we provide some sufficient conditions for consensus or for divergence. Besides the stability analysis, we also conduct some analysis for convergence characteristics in terms of locations of eigenvalues. Through a number of numerical simulations, we would attempt to understand the behaviors of misaligned consensus dynamics. 
\end{abstract}

\section{Introduction}
Consensus for multi-agent systems has been so widely studied for the last two decades \cite{jadbabaie2003coordination,olfati2007consensuspieee}. The consensus algorithms have been shown to be useful for various engineering applications such as mobile dispatch \cite{Tichakorn_icca_2010},  energy coordination in smart building \cite{Kim_2015_TCST}, smart grid \cite{Ziang_tps_2012}, and so on. There also have been so various research efforts in terms of theoretical developments. According to dynamic models, there are a number of different analyses for continuous time, discrete-time, cluster consensus, homogeneous, heterogeneous systems, stochastic, lower-order, high-order systems, etc. Meanwhile, according to the network topologies, directed, undirected, switching, and balanced graphs have been studied. Depending upon characteristics of sampling, synchronous, asynchronous, event-triggered, and quantized-based consensus algorithms also have been investigated. 

Given a network system, that is described by a graph $\mathcal{G} =(\mathcal{V}, \mathcal{E})$ where $\mathcal{V} = \{1, \ldots, n\}$ is the set of agents and $\mathcal{E}$ is the set of connectivities, however, the majority of existing consensus algorithms uses state information directly for control update. That is, when agent is modeled as $\dot{p}_i  = u_i(p_i, p_j),~j \in \mathcal{N}_i$, where $\mathcal{N}_i$ is the set of neighboring agents of agent $i$, for the control update $u_i(p_i, p_j)$, the agent $i$ uses the state information $p_i$ and $p_j$. Since the agent $i$ uses only its own state information $p_i$ along with neighboring information $p_j$, the consensus update $u_i(p_i, p_j)$ may be considered as decentralized. However, observing that the states $p_1, p_2, \ldots, p_n$ are all defined in a common global coordinate frame, we can notice that the control updates $u_1, u_2, \ldots, u_n$ are all defined in a common global coordinate frame. Thus, in this sense, most of existing consensus algorithms have been developed under the assumption of available global state information $p_i$ and $p_j$, or diffusively-coupling state information $p_i - p_j$. Note that the diffusive coupling state information $p_i - p_j$, where $(i,j) \in \mathcal{E}$, are defined in a common direction in general setups although it can be transformed into local frames (see (\ref{eq_distributed_measure}) in Section~\ref{section2}).

In this paper, we would like to study a consensus problem under the setup of misaligned information of $p_i - p_j$ that has been rarely examined in existing works. That is, in this paper, it is supposed that the common directions representing $p_i - p_j$  may be different according to the sensing or control capability of agents. We are mainly motivated to consider certain situations where agents' coordinate frames are misaligned, or sensing or control directions have been biased. In typical consensus algorithm given in (\ref{eq_consensus}), the summed diffusive coupling information $-\sum_{j \in \mathcal{N}_i} a_{ij} (p_i - p_j)$ are represented in a common direction, in ideal situations. However, it may be possible to imagine a circumstance where $-\sum_{j \in \mathcal{N}_i} a_{ij} (p_i - p_j)$ may be implemented in a wrong direction or biased direction, or the measurement $p_i - p_j$ may be biased. This paper seeks to find a consensus condition and attempts to understand the convergence characteristics under this circumstance.  
From literature search, it is observed that the consensus problem aforementioned has not been investigated directly; but some similar works have been studied. For examples, in \cite{weiren_cdc_2008,weiren_tac_2009}, the author studied a Cartesian coordinate coupling problem with a common coupling matrix $C$, and in \cite{Ramirez-Riberos_jgcd_2010}, they also studied a consensus problem with coupling multiplied by rotation matrix in $3$-D for cyclic formations. However, in \cite{weiren_cdc_2008,Ramirez-Riberos_jgcd_2010}, they did not consider a general case when agents have different misaligned orientation angles.  In distributed formation control, the orientation alignment problem has been key issues \cite{OhAhn_tac_2014,LeeAhn_automatica_2016}. It was shown that when orientations of agents are aligned, the desired formation could be achieved. Related with the orientation misalignment, a consensus with pursuit weight \cite{weiding_automatica_2010}, i.e., $k_{ij} e^{j \alpha_{ij}}$, may have some relevance. But, the pursuit system studied in \cite{weiding_automatica_2010} updates the control law after collecting the pursuit angles at each agent. In a different setup, which can be considered as orientation misalignment, rotation matrices were combined into the coupling terms in \cite{bhlee_icarcv_2016}. However, the ideas used in  \cite{weiding_automatica_2010,bhlee_icarcv_2016} do not consider the misalignment of control actions. Thus, since the misalignment angle is defined as the same one for each agent in our problem, the problems studied in \cite{weiding_automatica_2010,bhlee_icarcv_2016} are different from our current work. 
 
Consequently, the main contributions of this paper can be summarized as follows. First, we formulate a consensus problem with misaligned orientation angles  in which agents may implement control actions into misaligned directions, or update the consensus algorithms with misaligned measurements. We provide some conditions for consensus. Second, we also further conduct an analysis to estimate the locations of eigenvalues that are closely related with the convergence characteristics. Then, based on analyses, we conduct numerical simulations to understand the behaviors of agents in a better way. As far as the authors are concerned, even though there have been some related works, there has been no direct research for consensus under misaligned orientations.

This paper consists of as follows. In Section~\ref{section2}, we would like to formulate the consensus problem under misaligned orientations in a clear way. Then, in Section~\ref{section_analysis}, we provide analyses for consensus conditions and some related convergence characteristics. We present numerical illustrations in Section~\ref{section_examples} to validate the analyses and provide further discussions in Section~\ref{section_dis}. Conclusions will be presented in Section~\ref{section_conc}.


\section{Problem Formulation}\label{section2}
Let the global coordinate frame be denoted as $^g \Sigma$ and the $i$-th local coordinate frame as $^i \Sigma$. The position of agent $i$ is represented as $p_i \in \Bbb{R}^2$ in the global coordinate frame $^g \Sigma$, where $i \in \{1, \ldots, n\}$. Given two position vectors $p_i$ and $p_j$, the displacement vector between $p_i$ and $p_j$ is denoted as $z_{ij}= p_i - p_j$. The consensus algorithm for continuous-time linear systems is given as:
\begin{align}
\dot{p}_i &= -\sum_{j \in \mathcal{N}_i} a_{ij} (p_i - p_j) =  -\sum_{j \in \mathcal{N}_i} a_{ij} z_{ij}  \label{eq_consensus}
\end{align}
where $a_{ij} = a_{ji} \in \{0, 1\}$ depending upon the connectivity of undirected graphs. 
The consensus problem (\ref{eq_consensus}) can be rewritten as 
\begin{align}
{^{i}\dot{p}}_i &= -\sum_{j \in \mathcal{N}_i} a_{ij} ({^{i}p}_i - {^{i}p}_j) = \sum_{j \in \mathcal{N}_i} a_{ij} {^{i}p}_j \label{eq_distributed_measure}
\end{align}
which controls the movement of agent in its own local coordinate frame $^i \Sigma$. Obviously, the convergence property and stability  of (\ref{eq_consensus}) and (\ref{eq_distributed_measure}) are equivalent. However, the vector $z_{ij}$ is expressed in the global coordinate frame $^g \Sigma$, while the vector ${^{i}p}_j$ is expressed in the local coordinate frame $^i \Sigma$. Hence, the update of consensus algorithm could be done in two different scenarios. The first scenario, described in (\ref{eq_consensus}), is to use the vector $z_{ij}$ in orientation aligned coordinate frame, which is the local coordinate frame aligned to the global frame. In Fig.~\ref{problem_formulation}, the orientation aligned coordinate frames are denoted as ${^{g^i} \Sigma}$ (i.e., in this figure, they are $^{g^1} \Sigma$, $^{g^2} \Sigma$, and $^{g^3} \Sigma$ in red color). Thus, in the first approach, the orientation needs to be aligned. The second scenario, described in (\ref{eq_distributed_measure}), is to use ${^{i}p}_j$ in the local coordinate frame $^i \Sigma$, whose orientation is not aligned to the global frame. In Fig.~\ref{problem_formulation}, the local coordinate frames are denoted as $^1 \Sigma$, $^2 \Sigma$, and $^3 \Sigma$ in blue color. The first scenario (\ref{eq_consensus}) requires a constraint of alignment, while the second scenario (\ref{eq_distributed_measure}) is free from this constraint. Thus, the second scenario is considered to be of more distributed in the sense that the vector ${^{i}p}_j$ can be defined in local coordinate frames directly. It may be worthy of considering another scenario as follows: the local coordinate frames are not aligned; but agents have measured the positions $p_i$ in global coordinate frame (ex. using GPS information). In such case, they have to use $z_{ij}$ in the global coordinate frame; but since the orientations are not aligned to $^g \Sigma$, they need to sense the orientation angles or they need to find the direction of $^g \Sigma$ to calculate the orientation angles. In Fig.~\ref{problem_formulation}, the orientations are denoted as $\phi_i$. Then, using the obtained orientation angles $\phi_i$, agents can virtually rotate its local coordinate frame to global coordinate frame. 
Under this scenario, the coordinate frames ${^{g^i} \Sigma}$ in Fig.~\ref{problem_formulation} can be considered virtually aligned global coordinate frame.\footnote{Note that with available orientation angles $\phi_i$, the vectors $z_{ij}$ and ${^{i}p}_j$ are related by
\begin{align}
{^{i}p}_i - {^{i}p}_j  =  D(-\phi_i)(p_i - p_j)  = D(-\phi_i) z_{ij} = - {^{i}p}_j \label{eq_local_global}
\end{align}
where $ D(-\phi_i)$ is the rotation matrix. Thus, in terms of analysis, the three scenarios are equivalent.}
Then, the agent can implement the measurement $p_i - p_j = z_{ij}$ into the virtual coordinate frame ${^{g^i} \Sigma}$ for updating the consensus algorithm in global coordinate frame. This scenario may be considered the third scenario. 
Once again, note that the three scenarios have the same convergence and stability properties. In literature, there have been no distinctions between these three scenarios; however, in terms of sensing or in terms of implementation, they are essentially  different and should be distinguished. 
\begin{figure*}[t]
\centering
\includegraphics[width=8cm]{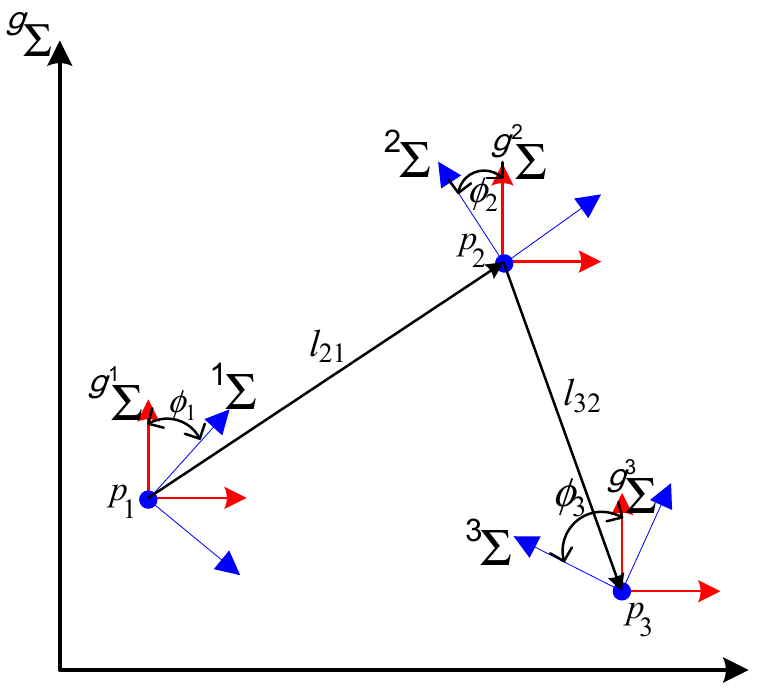}
\caption{Global coordinate frame $^g \Sigma$, local coordinate frames $^i \Sigma$, virtually aligned global coordinate frames ${^{g^i} \Sigma}$, and orientation angles $\phi_i$.} \label{problem_formulation}
\end{figure*}

As the main motivation of this paper, let us suppose that there are some errors in the orientation alignment, or orientation estimation in these three scenarios. Under the first scenario, the orientations of aligned local coordinate frames $^{g^i} \Sigma$, which are supposed to be aligned to the global frame $^g \Sigma$, may be not aligned as illustrated in Fig.~\ref{cood_virtual-1}. That is, the local coordinate frame $^{g^i} \Sigma$, which is supposed to be aligned to $^g \Sigma$, has in fact misalignment error as much as $\theta_i$. Under the second scenario, the agent $i$ may measure the direction of neighboring agent $j$ with angle error $\theta_i$. In Fig.~\ref{cood_virtual-2}, the agent $1$ measures the agent $2$ with direction error $\theta_1$. This can be interpreted as an error in control direction. Even though the displacement measurement is correct, during the implementation process, the agent may provide control command to wrong direction with orientation bias $\theta_i$. Under the third scenario, it can be assumed that the virtually aligned global coordinate frames have misalignment errors as $\theta_i$ as also illustrated in Fig.~\ref{cood_virtual-3}. Then, the third scenario would have the same alignment problem as the first scenario.

\begin{figure*}[t]
\centering
\subfigure[Misalignment orientation errors in scenario $1$.]{\includegraphics[width=5cm]{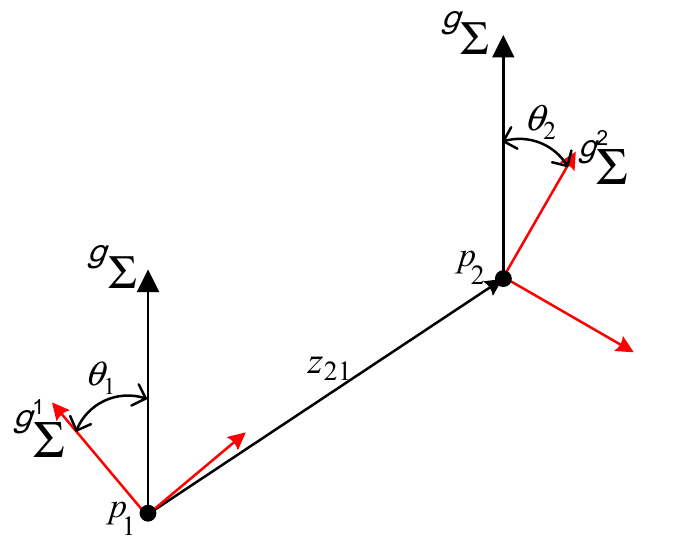}\label{cood_virtual-1}}
\subfigure[Misaligned orientation biases in scenario $2$.]{\includegraphics[width=5cm]{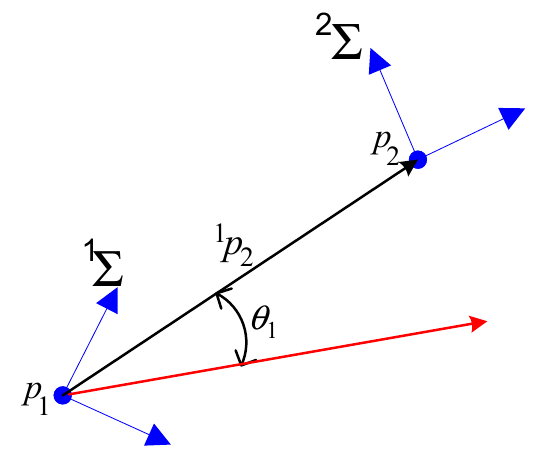}\label{cood_virtual-2}}
\subfigure[Misalignment orientation errors in scenario $3$.]{\includegraphics[width=5.1cm]{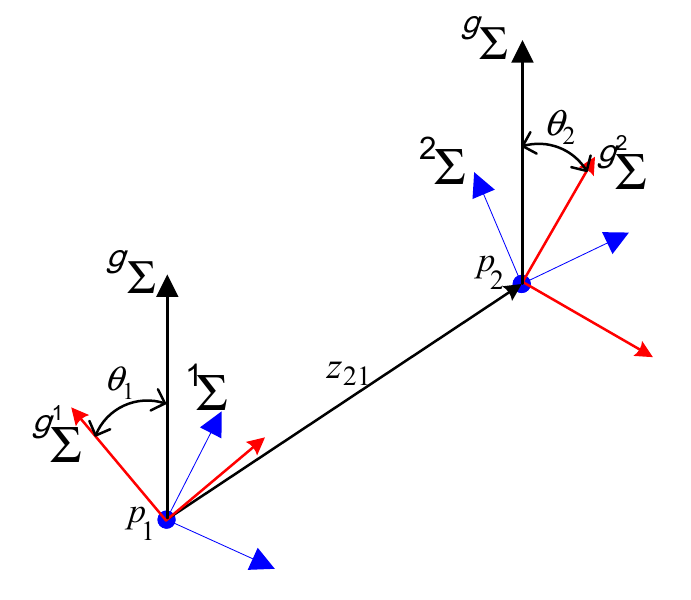}\label{cood_virtual-3}}\caption{Misalignment errors in three scenarios.}\label{cood_virtual}
\end{figure*}

Without notational confusion, the orientation errors or biases considered in the above three scenarios are called \textit{misaligned orientation errors}. The misaligned orientation errors may be due to errors in alignment to the global coordinate frame, biases in local sensing or in control direction, or errors in virtual alignment.  
With the presence of the error $\theta_i$, during the implementation, it can be represented in wrong directions. For example, it may be represented in the misaligned virtual coordinate frames as depicted in Fig.~\ref{cood_virtual-3}.  Thus, the control inputs transformed to the global coordinate frame $^g\Sigma$ would be defined as
\begin{align}
\dot{p}_i &= -D(\theta_i) \sum_{j \in \mathcal{N}_i} a_{ij} (p_i - p_j) \nonumber\\
           &= -\sum_{j \in \mathcal{N}_i} a_{ij} D(\theta_i)(p_i - p_j)  \label{eq_consensus_misaligned}
\end{align}
Note that in the above equation, $D(\theta_i)$ is a $SO(2)$ rotation matrix and the graph can be considered as directed, although the neighboring agents $i$ and $j$ mutually sense each other and exchange the information, since the weights for edges have the relationship $a_{ij} D(\theta_i) \neq  a_{ji} D(\theta_j)$ even though $a_{ij} = a_{ji}$. Then if we consider $a_{ij} D(\theta_i)$ as edge weighing, it is a matrix-weighted directed graph \cite{MinhAhn_arXiv2017,Minh_auto_submitted2017}. Let us define adjacency edge-matrix as $A_{ij} =  a_{ij} D(\theta_i)$. Then, the adjacency matrix is given as
\begin{align}
\mathcal{A} &=[ A_{ij} ] \nonumber\\
  &= \left[\begin{array}{ccccc}
                       0 & a_{12} D(\theta_1) & a_{13}D(\theta_1) & \cdots & a_{1n}D(\theta_1) \\
                       a_{21} D(\theta_2)& 0 & a_{23} D(\theta_2) & \cdots & a_{2n}D(\theta_2) \\
                       \vdots &    &          &    &  \vdots  \\
                       a_{n1} D(\theta_n) & \cdots  &   &  &  0    
                       \end{array}  \right] \in \Bbb{R}^{2n \times 2n}
\end{align}
For agent $i$, we define out-degree matrix as $D_i^{out} = \sum_{j \in \mathcal{N}_i} a_{ij} D(\theta_i)$ and in-degree matrix as $D_i^{in} = \sum_{j \in \mathcal{N}_i} a_{ji} D(\theta_j)$. Sincerely it is clear that ${D}_i^{out} \neq {D}_i^{in}$, it is not balanced. Then, we can define a block out-degree matrix of $\mathcal{G}$ as $\mathcal{D}^{out} = \text{blkdiag} [  D_i^{out}  ]$. Then, the SO(2)-weighted Laplacian matrix can be generated as
\begin{align}
\mathcal{L}^{out} = \mathcal{D}^{out} -\mathcal{A}
\end{align}
Using the above Laplacian matrix, we can have the state propagation as 
\begin{align} \label{eq_laplacian}
\dot{p} = - \mathcal{L}^{out} p(t)
\end{align}
where $p=(p_1^T, p_2^T, \ldots, p_n^T)^T \in \Bbb{R}^{2n}$.

\section{Convergence and stability analysis} \label{section_analysis}
This section is dedicated to the convergence and stability analysis of the system (\ref{eq_laplacian}). The following lemmas are developed for convergence analysis.

\begin{lemma} \label{lemma_null}
The Laplacian matrix $\mathcal{L}^{out}$ has rank as $\text{rank}(\mathcal{L}^{out})= 2n -2$, and the null space is given as $\mathcal{N}(\mathcal{L}^{out})= \text{span}\{ {\underbrace{[1,0, 1,0, \cdots, 1, 0]}_{\triangleq {\bf}{1}^e \in \Bbb{R}^{2n}}}^T, {\underbrace{ [0, 1, 0, 1, \cdots, 0, 1]}_{\triangleq {\bf}{1}^o \in \Bbb{R}^{2n}}}^T\}$.
\end{lemma}
\begin{proof}
Defining $q = c_1  {\bf}{1}^e + c_2  {\bf}{1}^o \triangleq q^e  + q_o$, where $c_1$ and $c_2$ are constants, we have $\mathcal{L}^{out} q =0$. Let us examine the converse. 
From (\ref{eq_consensus_misaligned}), equalizing $\dot{p}_i =0$, we will show that $\sum_{j \in \mathcal{N}_i} a_{ij} D(\theta_i)(p_i - p_j)=0$ only when $p_i = p_j$. Denoting $p_i = (x_i, y_i)^T$, we can write 
\begin{align}
\sum_{j \in \mathcal{N}_i} a_{ij} D(\theta_i)(p_i - p_j) &= \left[\begin{array} {cc}
                                                            \cos\theta_i & -\sin \theta_i \\
                                                            \sin\theta_i & \cos\theta_i \end{array} \right]      
\sum_{j \in \mathcal{N}_i} a_{ij} \left[\begin{array} {c}
                                   x_i - x_j \\
                                   y_i  - y_j 
                                   \end{array}\right] \nonumber\\
                          &=   \left[\begin{array} {cc}
                                                            \cos\theta_i & -\sin \theta_i \\
                                                            \sin\theta_i & \cos\theta_i \end{array} \right]      
 \left[\begin{array} {c}
                                  \sum_{j \in \mathcal{N}_i} a_{ij}  (x_i - x_j) \\
                                  \sum_{j \in \mathcal{N}_i} a_{ij} (y_i  - y_j) 
                                   \end{array}\right]      
\end{align}
Since the matrix $D(\theta_i)$ is non-singular, in order to get $\dot{p}_i =0$, we only need to have $\sum_{j \in \mathcal{N}_i} a_{ij}  (x_i - x_j) =0$ and 
$\sum_{j \in \mathcal{N}_i} a_{ij} (y_i  - y_j) =0$, which is the consensus in $x$-component and $y$-component respectively. Consequently, it proves that $\text{rank}(\mathcal{L}^{out})= 2n -2$.
\end{proof}

\begin{lemma}\label{lemma_zeroeig}
The set of eigenvalues of $\mathcal{L}^{out}$ contains only two zero eigenvalues corresponding to eigenvectors $\text{span}\{{\bf}{1}^e\}$ and $\text{span}\{{\bf}{1}^o\}$ respectively.
\end{lemma}
\begin{proof}
The Laplacian matrix is given as
\begin{align}
\mathcal{L}^{out} &= \left[\begin{array}{ccccc}
                       \sum_{j \in \mathcal{N}_1} a_{1j} D(\theta_1) & -a_{12} D(\theta_1) & -a_{13}D(\theta_1) & \cdots & -a_{1n}D(\theta_1) \\
                       -a_{21} D(\theta_2)& \sum_{j \in \mathcal{N}_2} a_{2j} D(\theta_2)   & -a_{23} D(\theta_2) & \cdots & -a_{2n}D(\theta_2) \\
                       \vdots &    &          &    &  \vdots  \\
                       -a_{n1} D(\theta_n) & \cdots  &   &  &  \sum_{j \in \mathcal{N}_n} a_{nj} D(\theta_n)    
                       \end{array}  \right]
\end{align}
Defining $m_i= \sum_{j=1}^n a_{ij}$, we can see that for each row vector, it has its diagonal component as $m_i \cos\theta_i$ and off diagonal components as $-m_i \sin\theta_i$, $a_{ij} \sin\theta_i , $ and $-a_{ij}\cos\theta_i$, $\forall j\in \mathcal{N}_i$, or $m_i \sin\theta_i$, $-a_{ij} \sin\theta_i , $ and $a_{ij}\cos\theta_i$, $\forall j\in \mathcal{N}_i$. Thus, the summation of elements of each row vector is zero, which completes the proof.
\end{proof}

It seems to be not trivial to show that all the eigenvalues of $\mathcal{L}^{out}$ are on the right-half plane (RHP) except two zero eigenvalues when $0 < \cos\theta_i \leq 1$. 
The following discussion and theorem are for this result. Let us decompose $\mathcal{L}^{out}$ as
\begin{align}
\mathcal{L}^{out} &= \left[\begin{array}{ccccc}
                        D(\theta_1) & 0 & 0 & \cdots & 0 \\
                       0 &  D(\theta_2)   & 0& \cdots & 0 \\
                       \vdots &    &    \ddots      &    &  \vdots  \\
                       0& \cdots  &   &  &   D(\theta_n)    
                       \end{array}  \right] \nonumber\\
                       & \times \left[\begin{array}{ccccc}
                       \sum_{j \in \mathcal{N}_1} a_{1j} I_2  & -a_{12} I_2 & -a_{13}I_2 & \cdots & -a_{1n}I_2 \\
                       -a_{21}I_2& \sum_{j \in \mathcal{N}_2} a_{2j}I_2  & -a_{23}I_2  & \cdots & -a_{2n}I_2 \\
                       \vdots &    &        \ddots    &    &  \vdots  \\
                       -a_{n1}I_2& \cdots  &   &  &  \sum_{j \in \mathcal{N}_n} a_{nj} I_2
                       \end{array}  \right] \nonumber\\
                       &\triangleq \text{blkdiag}[D(\theta_i)] \cdot \mathcal{L}^{o} \triangleq D_{D(\theta_i)} \cdot \mathcal{L}^{o}
\end{align}
where $D_{D(\theta_i)}$ is a block diagonal matrix and $\mathcal{L}^{o} = {\mathcal{L}^{o}}^T$ is the Laplacian matrix characterizing the topology only (let us call it \textit{topology Laplacian matrix}).  With the above decomposition, the dynamics (\ref{eq_laplacian}) is rewritten as
\begin{align} \label{eq_lzeta_maineq}
\dot{p} = - ( D_{D(\theta_i)} \mathcal{L}^{o} ) p(t)
\end{align}

\begin{theorem} Let $0 < \cos\theta_i \leq 1$. Then, the system described by (\ref{eq_laplacian}) is globally asymptotically stable to a consensus value and eigenvalues of $\mathcal{L}^{out}$ have positive-real parts except two zero eigenvalues. \label{theorem_consensus_main}
\end{theorem}
\begin{proof}It is clear that the topology Laplacian matrix $\mathcal{L}^{o}$ has two eigenvectors $u_1$ and $u_2$ correspoding to two zero eigenvalues. Let the topology Laplacian matrix $\mathcal{L}^{o}$ be decomposed as $U^T \mathcal{L}^{o} U = \text{diag} (\lambda_1, \lambda_2, \ldots, \lambda_{2n})$
where $\lambda_1 = \lambda_2 = 0 < \lambda_3 =  \lambda_4 \leq \ldots \lambda_{2k-1} = \lambda_{2k} \ldots   \leq \lambda_{2n-1} = \lambda_{2n}$ and $U =[u_{1}, u_{2}, u_{3}, \ldots, u_{2n-1}, u_{2n}] \in \Bbb{R}^{2n \times 2n}$ is the orthogonal matrix composed of eigenvectors of $\mathcal{L}^{o}$. Thus, we have $u_1= c_3 q^e$ and $u_2 = c_4 q^o$ where $c_3$ and $c_4$ are constants, and $U^{-1} = U^T$. Let us define another matrix $U' =[u_{3}, \ldots, u_{2n-1}, u_{2n}]\in \Bbb{R}^{2n \times (2n-2)}$. Then, we can have
\begin{align}
U^T \mathcal{L}^{o} U = \left[ \begin{array}{cc}
                              0_{2 \times 2} & 0_{2 \times {2n-2}} \\
                              0_{{2n-2} \times 2} &  (U')^T \mathcal{L}^{o} U' \\
                              \end{array} \right]
\end{align}
It is clear that $(U')^T \mathcal{L}^{o} U'$ is positive definite. Let $p' =  (U')^T p  \in \Bbb{R}^{2n-2}$. Next, let us select a Lyapunov candidate as
\begin{align}
V = (p')^T (U')^T \mathcal{L}^{o} U' p'
\end{align}
The derivative is obtained as $\dot{V} = -(\dot{p}')^T (U')^T \mathcal{L}^{o} U' p' -  ({p}')^T (U')^T \mathcal{L}^{o} U' \dot{p}' = -(\dot{p})^T U' (U')^T \mathcal{L}^{o} U' (U')^T p $  $  -({p})^T  U' (U')^T \mathcal{L}^{o} U' (U')^T \dot{p}$. It can be shown that $U' (U')^T \mathcal{L}^{o} U' (U')^T = \mathcal{L}^{o}$. Let $[u_1, u_2]= z$ and $[u_1, u_2]^T= z^T$. Then, from $U= [z, U']$, we have $U U^T = z z^T + U' {U'}^T = I_{2n \times 2n}$, which means that $U' {U'}^T = I_{2n \times 2n} - z z^T$. Hence $U' (U')^T \mathcal{L}^{o} U' (U')^T = ( I_{2n \times 2n} - z z^T ) \mathcal{L}^{o} ( I_{2n \times 2n} - z z^T ) = ( I_{2n \times 2n} - z z^T ) (\mathcal{L}^{o}  - \mathcal{L}^{o} z z^T ) = ( I_{2n \times 2n} - z z^T ) \mathcal{L}^{o} = \mathcal{L}^{o} - z z^T \mathcal{L}^{o} = \mathcal{L}^{o}$ due to $\mathcal{L}^{o} z =0$ and $z^T \mathcal{L}^{o}=0$.
Therefore, we obtain
\begin{align}
\dot{V}  &= - (\dot{p})^T \mathcal{L}^{o} p - ({p})^T \mathcal{L}^{o} \dot{p} \nonumber\\ 
         &= -p^T {\mathcal{L}^{o}}^T D_{D(\theta_i)}^T \mathcal{L}^{o} p   -  ({p})^T \mathcal{L}^{o} D_{D(\theta_i)} \mathcal{L}^{o} p \nonumber\\
         & = -(\mathcal{L}^{o} p)^T \left[ D_{D(\theta_i)}^T  + D_{D(\theta_i)}\right] (\mathcal{L}^{o} p) \leq 0
\end{align}
The last inequality holds since $D_{D(\theta_i)}^T  + {D(\theta_i)} >0$ if  $0 < \cos\theta_i \leq 1$. It is now clear that $(\mathcal{L}^{o} p)^T ( D_{D(\theta_i)}^T  + D_{D(\theta_i)} ) (\mathcal{L}^{o} p) = 0$ if and only if $\mathcal{L}^{o} p=0$, which implies that $p$ converges to a consensus and does not diverge because of $\dot{p}=0$. 
\end{proof}

The above \textit{Theorem~\ref{theorem_consensus_main}} provides the condition for the consensus when $0 < \cos\theta_i \leq 1$. Let some $\theta_i$ be negative as $\theta_i <0$. Then some diagonal elements of $D_{D(\theta_i)}^T  + D_{D(\theta_i)}$ will be negative. However, even with some negative diagonal elements of $D_{D(\theta_i)}^T  + D_{D(\theta_i)}$, $\dot{V} >0$ is not ensured. Thus, even with some negative $\theta_i$, the instability is not ensured. Also, the \textit{Theorem~\ref{theorem_consensus_main}} does not provide the locations of eigenvalues, which is related with the convergence characteristics. To evaluate the location of eigenvalues, we use Gershgorin circle for block matrix \cite{Feingold_PJM_1962,Sluis_LAA_1979}, which is summarized as follows:
\begin{lemma} \label{Gershgorin_lemma} Given a block matrix $A=[A_{ij}]$, all eigenvalues of $A$ are contained in the set $G= \sum_{i=1}^n \cup G_i$ where $G_i$ is the set of all $\lambda \in \Bbb{C}$ satisfying
\begin{align}
\Vert (A_{ii} - \lambda I_2)^{-1} \Vert^{-1} \leq \sum_{j=1, j\neq i}^n \Vert A_{ij} \Vert
\end{align}
\end{lemma}

\begin{theorem} \label{theorem_not_pd}
If agent $i$ is misaligned as $\cos\theta_i < 1$, then its eigenvalues are within the circles with center $(m_i \cos\theta_i, \pm m_i \sin \theta_i)$ and with radius $m_i$ in the complex domain, where $m_i$ is the cardinality of $\mathcal{N}_i$.
\end{theorem}
\begin{proof}
From \textit{Lemma~\ref{Gershgorin_lemma}}, we know that eigenvalues $\lambda_i$ of each row block matrix satisfy the following inequality:
\begin{align}
\left\Vert \left( \sum_{j \in \mathcal{N}_i} a_{ij} D(\theta_i)- \lambda_i I_2 \right)^{-1} \right\Vert^{-1} &\leq \sum_{j \in \mathcal{N}_i, j \neq i}^n \Vert a_{ij} D(\theta_i) \Vert \nonumber\\
\Longleftrightarrow \left\Vert \left( m_i D(\theta_i)- \lambda_i I_2 \right)^{-1} \right\Vert^{-1} &\leq m_i \Vert D(\theta_i) \Vert \label{main_condition}
\end{align}
Since singular value of a SO(2) matrix is $1$, we can change the above inequality as
\begin{align}
\left\Vert \left( D(\theta_i)- \frac{\lambda_i}{m_i} I_2 \right)^{-1} \right\Vert^{-1} &\leq 1 \nonumber\\
\Longleftrightarrow \left\Vert \left( D(\theta_i)- \frac{\lambda_i}{m_i} I_2 \right)^{-1} \right\Vert &\geq 1
\end{align}
The inverse of the left-hand side can be obtained as:
\begin{align}
\left(D(\theta_i)- \frac{\lambda_i}{m_i} I_2 \right)^{-1} =\underbrace{ \frac{1}{  \frac{1}{m_i^2}\lambda_i^2  - \frac{2 \cos\theta_i}{m_i}\lambda_i  + 1  }   \left[\begin{array}{cc}
\cos\theta_i -\frac{\lambda_i}{m_i}   &  \sin\theta_i  \\
-\sin\theta_i    & \cos\theta_i -\frac{\lambda_i}{m_i} 
             \end{array}\right] }_{\triangleq P} \label{inverse_diagonal}
\end{align}
Denote $\frac{\lambda_i}{m_i} \triangleq \kappa_i = \alpha_i + j \beta_i$ and $\frac{\lambda_i^\ast}{m_i} \triangleq \kappa_i^\ast = \alpha_i - j \beta_i$, where $\lambda_i^\ast$ is the conjugate of $\lambda_i$. Then, in order to use the relationship $\Vert P \Vert = \sqrt{ \lambda_{max}(P^\ast P) }$, where $P^\ast$ is a complex conjugate of $P$, $P^\ast P$ can be obtained as:
\begin{align}
P^\ast P &= \frac{1}{( {\kappa^\ast}_i^2 - 2 \cos\theta_i {\kappa}_i^\ast +1 )  ( {\kappa}_i^2 - 2 \cos\theta_i {\kappa}_i +1 )   } \nonumber\\
& \times  \underbrace{\left[\begin{array}{cc}
\cos\theta_i -\kappa_i^\ast   &  -\sin\theta_i  \\
\sin\theta_i    & \cos\theta_i -\kappa_i^\ast
             \end{array}\right] \left[\begin{array}{cc}
\cos\theta_i -\kappa_i   &  \sin\theta_i  \\
-\sin\theta_i    & \cos\theta_i -\kappa_i
             \end{array}\right] }_{ \triangleq  \mathbf{T}}
\end{align}
It is noticeable that ${\kappa_i^\ast}^2 - 2 \cos\theta_i \kappa_i^\ast +1$ and ${\kappa}_i^2 - 2 \cos\theta_i {\kappa}_i +1$ can be decomposed as 
\begin{align}
{\kappa_i^\ast}^2 - 2 \cos\theta_i \kappa_i^\ast +1 &= ( \kappa_i^\ast -\cos\theta_i - j \sin\theta_i )  ( {\kappa}_i^\ast -\cos\theta_i + j \sin\theta_i ) \nonumber\\
{\kappa}_i^2 - 2 \cos\theta_i {\kappa}_i +1 &=  ( {\kappa}_i -\cos\theta_i + j \sin\theta_i )  ( {\kappa}_i -\cos\theta_i - j \sin\theta_i ) \nonumber
\end{align}
Then, $({\kappa_i^\ast}^2 - 2 \cos\theta_i \kappa_i^\ast +1 )  ( {\kappa}_i^2 - 2 \cos\theta_i {\kappa}_i +1)$ can be rewritten as
\begin{align}
& ({\kappa_i^\ast}^2 - 2 \cos\theta_i {\kappa_i^\ast} +1 )  ( {\kappa}_i^2 - 2 \cos\theta_i {\kappa}_i +1) \nonumber\\
&~~~~~ = ( {\kappa_i^\ast} -\cos\theta_i - j \sin\theta_i ) ( {\kappa_i^\ast} -\cos\theta_i + j \sin\theta_i ) \nonumber\\
&~~~~~~~~~~ \times  ( {\kappa}_i -\cos\theta_i + j \sin\theta_i )  ( {\kappa}_i -\cos\theta_i - j \sin\theta_i ) \nonumber\\
&~~~~~= ( {\kappa_i^\ast} -\cos\theta_i - j \sin\theta_i )  ( {\kappa}_i -\cos\theta_i + j \sin\theta_i ) \nonumber\\
&~~~~~~~~~~ \times    ( {\kappa_i^\ast} -\cos\theta_i + j \sin\theta_i )  ( {\kappa}_i -\cos\theta_i - j \sin\theta_i ) \nonumber\\
&~~~~~= \underbrace{(  {\kappa_i^\ast} \kappa_i - {\kappa_i^\ast} \cos\theta_i + j {\kappa_i^\ast} \sin\theta_i - \kappa_i \cos\theta_i - j \kappa_i \sin\theta_i +1)}_{\triangleq \xi_i^\ast } \nonumber \\
&~~~~~~~~~~ \times \underbrace{( {\kappa_i^\ast} \kappa_i - {\kappa_i^\ast} \cos\theta_i - j {\kappa_i^\ast} \sin\theta_i - \kappa_i \cos\theta_i + j \kappa_i \sin\theta_i +1)}_{\triangleq \xi_i}
\end{align}
Meanwhile the eigenvalues of $\mathbf{T}$ are obtained as $1 - \kappa_i \cos\theta_i - {\kappa_i^\ast} \cos\theta_i + {\kappa_i^\ast}\kappa_i \pm j(\kappa_i \sin\theta_i - {\kappa_i^\ast} \sin\theta_i )$, which are conjugate pair and equivalent to $\xi_i$ or $\xi_i^\ast$. Furthermore, we can see that $1 - \kappa_i \cos\theta_i - {\kappa_i^\ast}\cos\theta_i + {\kappa_i^\ast}\kappa_i \pm j(\kappa_i \sin\theta_i - {\kappa_i^\ast} \sin\theta_i ) = 1  - 2 \alpha_i \cos\theta_i + \alpha_i^2 + \beta_i^2 \pm 2 \beta_i \sin \theta_i = (\alpha_i  - \cos\theta_i)^2  + (\beta_i  \pm \sin\theta_i)^2 \geq 0$. Thus, when $\lambda_i = m_i \cos\theta_i \mp  j m_i \sin\theta_i$, $\xi_i^\ast = 0$ and  $\xi_i = 0$. Hence, except these cases, we have $\xi_i^\ast >0$ and $\xi_i >0$. Consequently, $\lambda_{max}(P^\ast P)$ can be either $\frac{1}{\xi_i}$ or $\frac{1}{\xi_i^\ast}$ because the eigenvalues of $T$ are the conjugate pair. To satisfy the condition $\Vert P \Vert \geq 1$, it is now required to have $\xi_i \leq 1$ or $\xi_i^\ast \leq 1$.

Since $\xi_i = (\alpha_i  - \cos\theta_i)^2  + (\beta_i  + \sin\theta_i)^2$ and $\xi_i^\ast = (\alpha_i  - \cos\theta_i)^2  + (\beta_i  - \sin\theta_i)^2$, the eigenvalues $\lambda_i$ of (\ref{main_condition}) need to satisfy 
$(\alpha_i  - \cos\theta_i)^2  + (\beta_i  + \sin\theta_i)^2 \leq 1$ or $(\alpha_i  - \cos\theta_i)^2  + (\beta_i  - \sin\theta_i)^2 \leq 1$. From $\frac{\lambda_i}{m_i} = \kappa_i = \alpha_i + j \beta_i$, replacing $\lambda_i = \text{Re}\lambda_i + j \text{Im}\lambda_i$, where $\text{Re}\lambda_i$ is the real part of $\lambda_i$ and $\text{Im}\lambda_i$ is the imaginary part of $\lambda_i$, we can have the following region for $\lambda_i$:
\begin{align}
\lambda_i \in \mathcal{R}^c &\triangleq \left\{\lambda :  (\text{Re}\lambda/m_i  - \cos\theta_i)^2  + (\text{Im}\lambda/m_i  + \sin\theta_i)^2 \leq 1               \right\} \nonumber\\
&~~~~~\bigcup \left\{\lambda :  (\text{Re}\lambda/m_i  - \cos\theta_i)^2  + (\text{Im}\lambda/m_i  - \sin\theta_i)^2 \leq 1    \right   \}\nonumber\\
& = \left\{\lambda :  (\text{Re}\lambda  - m_i \cos\theta_i)^2  + (\text{Im}\lambda  + m_i\sin\theta_i)^2 \leq m_i^2               \right\} \nonumber\\
&~~~~~ \bigcup \left\{\lambda :  (\text{Re}\lambda  - m_i \cos\theta_i)^2  + (\text{Im}\lambda  - m_i \sin\theta_i)^2 \leq m_i^2    \right   \} \label{eq_theorem1_cond}
\end{align}
which completes the proof.
\end{proof}

\begin{corollary} \label{corollary_not_negative2}
If $\cos\theta_i =1$, then all the eigenvalues of $\mathcal{L}^{out}$ are not negative, while if $\cos\theta_i = -1$, all the eigenvalues of $\mathcal{L}^{out}$ are not positive. 
\end{corollary}
\begin{proof}
When $\cos\theta_i =1$, the center of the circle defined by (\ref{eq_theorem1_cond}) is $(m_i, 0)$ with radius $m_i$. While, with $\cos\theta_i = -1$ the center of the circle is $(0, \pm m_i)$ with radius $m_i$. 
\end{proof}

It is noticeable that when $\cos\theta_i =1$ (i.e., $\theta_i =0$), the condition (\ref{main_condition}) can be changed as
\begin{align}
 \left\Vert \left( m_i I_2- \lambda_i I_2 \right)^{-1} \right\Vert^{-1} &\leq m_i \Vert I_2 \Vert \label{main_condition_0degree}
\end{align}
As clear by the above inequality and also by Corollary~\ref{corollary_not_negative2}, $\lambda_i$ should be ranged as $0 \leq \lambda_i \leq 2 m_i$ with real part only. Using the above result, we may find more precise region for eigenvalues when $0\leq \cos\theta_i \leq 1$, which is summarized in the following theorem. For the theorem, we use some properties of matrix norm such as
\begin{align}
\Vert D(\theta_i) A \Vert &= \sqrt{\lambda_{max}( A^\ast   D(\theta_i)^\ast D(\theta_i)  A )} = \sqrt{\lambda_{max}( A^\ast A )}  \nonumber\\
&~~~~~ = \Vert A \Vert = \Vert D(\theta_i) \Vert \Vert A \Vert~\text{and}~\Vert D(\theta_i) \Vert^{-1} = \Vert D(\theta_i)^{-1} \Vert
\end{align}
where $D(\theta_i) \in SO(2)$.

\begin{corollary} \label{corollary_equivalent_if}
The condition (\ref{main_condition_0degree}), which is for $\theta_i =0$, is satisfied if $\left\Vert \left( m_i D(\theta_i)- D(\theta_i)  \lambda_i I_2 \right)^{-1} \right\Vert^{-1} \leq m_i \Vert D(\theta_i) \Vert$ with $0 \leq \lambda_i \leq 2 m_i$.
\end{corollary}
\begin{proof}
From the following relationship, the \textit{if} condition is direct.
\begin{align}
&\left\Vert \left( m_i D(\theta_i)- D(\theta_i)  \lambda_i I_2 \right)^{-1} \right\Vert^{-1} \leq m_i \Vert D(\theta_i) \Vert \nonumber\\
&\Longrightarrow \left\Vert  D(\theta_i)^{-1}   \left( m_i I_2 - \lambda_i I_2 \right)^{-1}    \right\Vert^{-1} \leq m_i \Vert D(\theta_i) \Vert \nonumber\\
&\Longrightarrow  \Vert  D(\theta_i)^{-1} \Vert^{-1}     \left\Vert \left( m_i I_2 - \lambda_i I_2 \right)^{-1}    \right\Vert^{-1} \leq m_i \Vert D(\theta_i) \Vert \nonumber\\
&\Longrightarrow  \Vert  D(\theta_i) \Vert     \left\Vert \left( m_i I_2 - \lambda_i I_2 \right)^{-1}    \right\Vert^{-1} \leq m_i \Vert D(\theta_i) \Vert \nonumber\\
&\Longrightarrow    \left\Vert \left( m_i I_2 - \lambda_i I_2 \right)^{-1}    \right\Vert^{-1} \leq m_i \Vert I_2 \Vert \nonumber
\end{align}
\end{proof}
However, unfortunately, the \textit{only if} condition is not satisfied due to the following relationship, with $\psi_i \neq \theta_i$: 
\begin{align}
&\left\Vert \left( m_i D(\psi_i)- D(\psi_i)  \lambda_i I_2 \right)^{-1} \right\Vert^{-1} \leq m_i \Vert D(\theta_i) \Vert \nonumber\\
&\Longrightarrow \left\Vert  D(\psi_i)^{-1}   \left( m_i I_2 -  \lambda_i I_2 \right)^{-1}    \right\Vert^{-1} \leq m_i \Vert D(\theta_i) \Vert \nonumber\\
&\Longrightarrow    \left\Vert \left( m_i I_2 -  D(\psi_i)  \lambda_i I_2 \right)^{-1}    \right\Vert^{-1} \leq m_i \Vert I_2 \Vert \nonumber
\end{align}

From \textit{Theorem~\ref{theorem_consensus_main}}, it is now confirmed that the eigenvalues of $\mathcal{L}^{out}$ are on the right-half plane except two zero eigenvalues when $0 < \cos\theta_i \leq 1,~\forall i \in \mathcal{V}$ . Thus, by combining \textit{Theorem~\ref{theorem_consensus_main}} and \textit{Theorem~\ref{theorem_not_pd}}, we can make the following theorem:
\begin{theorem} \label{theorem_three_suff}
When $0 < \cos\theta_i \leq 1,~\forall i \in \mathcal{V}$, the eigenvalues of $\mathcal{L}^{out}$, i.e., $\lambda_i$, are placed as
\begin{align}
\lambda_i \in \mathcal{R}^c \bigcap \{\lambda: \text{Re}(\lambda)>0 \} \label{eq_theorem1_cond2}
\end{align}
while, when $-1 \leq \cos\theta_i < 0,~\forall i \in \mathcal{V}$, the eigenvalues are placed as 
\begin{align}
\lambda_i \in \mathcal{R}^c \bigcap \{\lambda: \text{Re}(\lambda)<0 \} \label{eq_theorem1_cond3}
\end{align}
Furthermore, when $\cos\theta_i =0,~\forall i \in \mathcal{V}$, all the eigenvalues are located on the imaginary axis. 
\end{theorem}
\begin{proof} The case of (\ref{eq_theorem1_cond2}) is direct by \textit{Theorem~\ref{theorem_consensus_main}} and \textit{Theorem~\ref{theorem_not_pd}}.
For the case of (\ref{eq_theorem1_cond3}), let $\frac{\pi}{2} < \theta_i \leq \pi$. Then, $\theta_i = \frac{\pi}{2} + \theta'$, where $0< \theta' \leq \frac{\pi}{2}$. Thus, we can have
\begin{align}
D(\theta_i) = D({\pi}/{2} + \theta')  &= \left[\begin{array} {cc}
 \cos({\pi}/{2} + \theta') & -\sin({\pi}/{2} + \theta') \\
\sin({\pi}/{2} + \theta')& \cos({\pi}/{2} + \theta') \end{array} \right] \nonumber\\
&=  \left[\begin{array} {cc}
- \sin(\theta') & -\cos(\theta') \\
\cos(\theta')& -\sin(\theta') \end{array} \right] = -\left[\begin{array} {cc}
\sin(\theta') & \cos(\theta') \\
-\cos(\theta')& \sin(\theta') \end{array} \right]
\end{align}
Likewise, when $-\pi \leq \theta_i < -\frac{\pi}{2}$, we have $\theta_i = -\frac{\pi}{2} + \theta'$, where $-\frac{\pi}{2} \leq \theta' < 0$. Thus, we have $D(\theta_i) = D(-{\pi}/{2} + \theta') =  \left[\begin{array} {cc}
 \sin(\theta') & \cos(\theta') \\
-\cos(\theta')& \sin(\theta') \end{array} \right]$. Thus, based on the analysis of \textit{Theorem~\ref{theorem_consensus_main}}, we can see that the signs of real parts of eigenvalues when $-1 \leq \cos\theta_i < 0,~\forall i \in \mathcal{V}$ are reversed from the signs of real parts of eigenvalues when $0 < \cos\theta_i \leq 1,~\forall i \in \mathcal{V}$. Lastly,  when $\cos\theta_i =0,~\forall i \in \mathcal{V}$, due to $D(\theta_i = \pi/2)= D(\theta_j= \pi/2), ~\forall i, j \in \mathcal{V}$, we can have $\mathcal{L}^{out} = D(\pi/2) \otimes \mathcal{L}^{o}$, where $\otimes$ is the Kronecker product. Using the eigenvalue property of Kronecker product, we can see that all of the eigenvalues of $\mathcal{L}^o$ are rotated by $\pm \pi/2$. Thus, when $\cos\theta_i =0,~\forall i \in \mathcal{V}$, eigenvalues are placed on the imaginary axis. 
\end{proof}

\begin{remark} \label{remark_1} Let us divide the nodes as $\mathcal{V} = \mathcal{V}_{+} \bigcup \mathcal{V}_{-}$, in which it holds $0 < \cos\theta_i \leq 1,~\forall i \in \mathcal{V}_{+}$ and $-1 \leq \cos\theta_i \leq  0$ $\forall i \in \mathcal{V}_{-}$. Let us suppose that $\mathcal{V}_{+}$ is a non-empty set and  $\mathcal{V}_{-}$ is also a non-empty set. In this case, it is hard to estimate the locations of eigenvalues analytically. From a number of numerical tests, in certain conditions, all the eigenvalues still could be located in the RHP, while in most cases, there were eigenvalues located in the LHP. 
\end{remark}

As a special case of \textit{Theorem~\ref{theorem_three_suff}}, when the misaligned orientation angles are equivalent as $\theta_i = \theta_j = \theta, ~i \neq j$ and $0 < \cos \theta_i = \cos \theta$, we can have the following result, which was also studied in Corollary 3.3 of \cite{weiren_cdc_2008}, 
\begin{corollary}
When $\theta_i = \theta_j = \theta, ~i \neq j$ and $0 < \cos \theta_i,~ \forall i$, the consensus is achieved, with eigenvalues of $\mathcal{L}^o$ rotated by angles $\theta$ and $-\theta$ respectively. 
\end{corollary}

\begin{remark} \label{remark_2} When $\theta_i \neq 0$, the Laplacian matrix $\mathcal{L}^{out}$ has complex eigenvalues; so the trajectories of agents may exhibit behaviors of stable or unstable focus, which requires more complicated convergence behaviors. Also the average  consensus is no more ensured due to $\sum_{i=1}^n \dot{p}_i =  \sum_{i=1}^n \left[ \sum_{j \in \mathcal{N}_i} a_{ij} D(\theta_i)(p_i - p_j) \right]\neq 0$ before converging to a common value. 
\end{remark}

Related with \textit{Remark~\ref{remark_2}}, we can see that the consensus point is a function of rotated initial positions of agents and magnitudes of rotation angles $\theta_i$. Let the topology Laplacian matrix be changed as $\mathcal{L}^{o} = \mathcal{L} \otimes I_2$, where $\mathcal{L}$ is the normal Laplacian matrix for a connected graph. Then, (\ref{eq_lzeta_maineq}) is changed as
\begin{align}
\dot{p} = - D_{D(\theta_i)} ( \mathcal{L} \otimes I_2 ) p(t) \label{eq_pdll}
\end{align}
The Laplacian matrix $\mathcal{L}$ satisfies $v^T \mathcal{L} =0$, where $v=(1, 1, \ldots, 1)^T \in \Bbb{R}^n$ is the vector with all elements being $1$. Now, we can have the following theorem.
\begin{theorem} \label{theorem_final_consensus}
Let $0< \cos \theta_i$ for all $i \in \mathcal{V}$, and let the initial positions be denoted as ${p}_i(0)$ and the final consensus point as ${p}(t_f)$. Then, ${p}(t_f)$ is computed as $p^{final} = Y(\theta_1, \ldots, \theta_n) \sum_{i=1}^n D(\theta_i)^T {p}_i(0)$ where  $Y(\theta_1, \ldots, \theta_n)$ is a $2 \times 2$ matrix that is a function of $\theta_i$. 
\end{theorem}
\begin{proof}
From (\ref{eq_pdll}), it follows that
\begin{align}
 (v^T \otimes I_2) D_{D(\theta_i)}^T   \dot{p} &= -(v^T \otimes I_2) D_{D(\theta_i)}^T D_{D(\theta_i)} ( \mathcal{L} \otimes I_2 ) p(t) \nonumber\\
 & = -(v^T \mathcal{L} \otimes I_2 ) p(t) =0   \label{eq_pdll_2}
\end{align}
which means that $(v^T \otimes I_2) D_{D(\theta_i)}^T   {p}(t_f) = (v^T \otimes I_2) D_{D(\theta_i)}^T   {p}(0) = \sum_{i=1}^n D(\theta_i)^T {p}_i(0)$. When a consensus is achieved as $p^{final}$ at $t = t_f$,  we have $p(t_f) = ( {p^{final}}^T,  {p^{final}}^T, \ldots,  {p^{final}}^T )^T$. Let $W \triangleq (v^T \otimes I_2) D_{D(\theta_i)}^T \in \Bbb{R}^{2 \times 2n}$. Then, we can write  $(v^T \otimes I_2) D_{D(\theta_i)}^T   {p}(t_f)$ at consensus as $(\sum_{i=1}^n  W_{1:2, 2i-1: 2i}) p^{final}$, where $W_{1:2, 2i-1: 2i}$ is the $2 \times 2$ matrix composing of $(2i-1)$-th and $2i$-th column vectors.  Thus, the consensus value is obtained as $p^{final} = (\sum_{i=1}^n  W_{1:2, 2i-1: 2i})^{-1} \sum_{i=1}^n D(\theta_i)^T {p}_i(0)$, where $\sum_{i=1}^n  W_{1:2, 2i-1: 2i}$ is nonsingular if $0< \cos \theta_i$ for all $i \in \mathcal{V}$. Note that $Y(\theta_1, \ldots, \theta_n)  =(\sum_{i=1}^n  W_{1:2, 2i-1: 2i})^{-1}$, which completes the proof.
\end{proof}

\section{Examples} \label{section_examples}

\subsection{Example $1$: Two agents}  \label{section_examples_sub1}
\begin{figure*}[t]
\centering
\includegraphics[width=12cm]{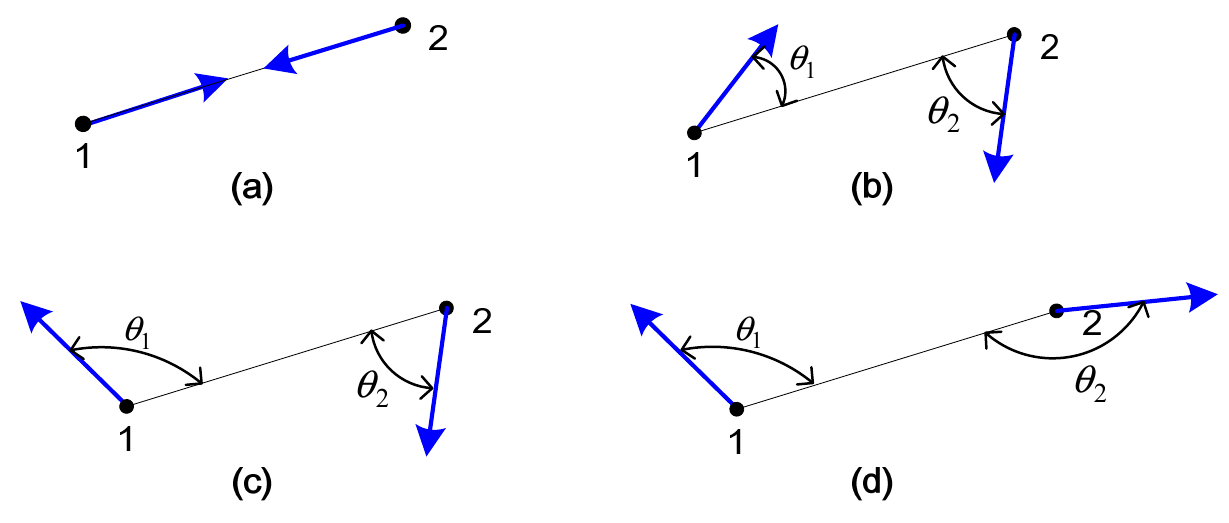}
\caption{Four case of example $1$: According to angles $\theta_1$ and $\theta_2$, the consensus may be not achieved; or the average consensus cannot be ensured.} \label{example1}
\end{figure*}

\subsubsection{Case $1$: With $\theta_1 = \theta_2 =0$} Let us consider the traditional consensus with $\theta_1 = \theta_2 =0$. Agents may try to reach each other along the line connecting two agents. This case is depicted in Fig.~\ref{example1}(a). The eigenvalues of $\mathcal{L}^{out}$ are $0, 0, 2, 2$. Fig.~\ref{example1_case1} shows the simulation results. As expected, the agents are approaching directly. 
\begin{figure*}[t]
\centering
\includegraphics[width=6cm]{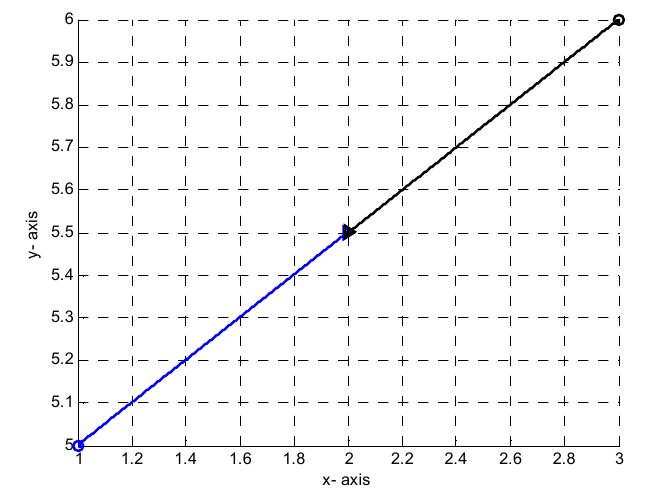}
\caption{Example $1$: Case $1$. The circles are initial positions and the triangles are final positions.} \label{example1_case1}
\end{figure*}

\subsubsection{Case $2$: With $0< \vert \theta_1 \vert, \vert \theta_2 \vert < \pi/2$}
Let us consider the case depicted in Fig.~\ref{example1}(b). With $\theta_1 = \pi/4$ and $\theta_2 =\pi/4$, we have eigenvalues as $0, 0,   1.414 \pm j 1.414$. Thus, the consensus will be achieved well. However, due to the pair of the complex conjugate eigenvalues, the convergence property would be more complicated. Meanwhile, with $\theta_1 = \pi/4$ and $\theta_2 = -\pi/4$, the eigenvalues are obtained as $0, 0, 1.414213562373095, 1.414213562373095$. In this case, since we have only real eigenvalues, the movements may be on the straight lines. But as expected in Fig.~\ref{example1}(b), the average consensus may be not achieved because agents do not move toward the center of the initial positions. As shown in the right-top of Fig.~\ref{example1_case2}, although a consensus is achieved, it is not average. Similarly, with $\theta_1 = \pi/4$ and $\theta_2 = -\pi/3$, we have the eigenvalues as $0, 0, 1.207106781186548 \pm j 0.158918622597891$. But, due to $\theta_1 \neq \theta_2$, the average consensus is not achieved as shown in the left-bottom of Fig.~\ref{example1_case2}. Likewise, with $\theta_1 = \pi/2.5$ and $\theta_2 = -\pi/2.2$ that gives the eigenvalues as $0, 0, 0.451331832648233 \pm j 0.038764925585779$, the consensus point is far away from the initial positions. It is drawn in the right-bottom of Fig.~\ref{example1_case2}.
\begin{figure*}[t]
\centering
\includegraphics[width=12cm]{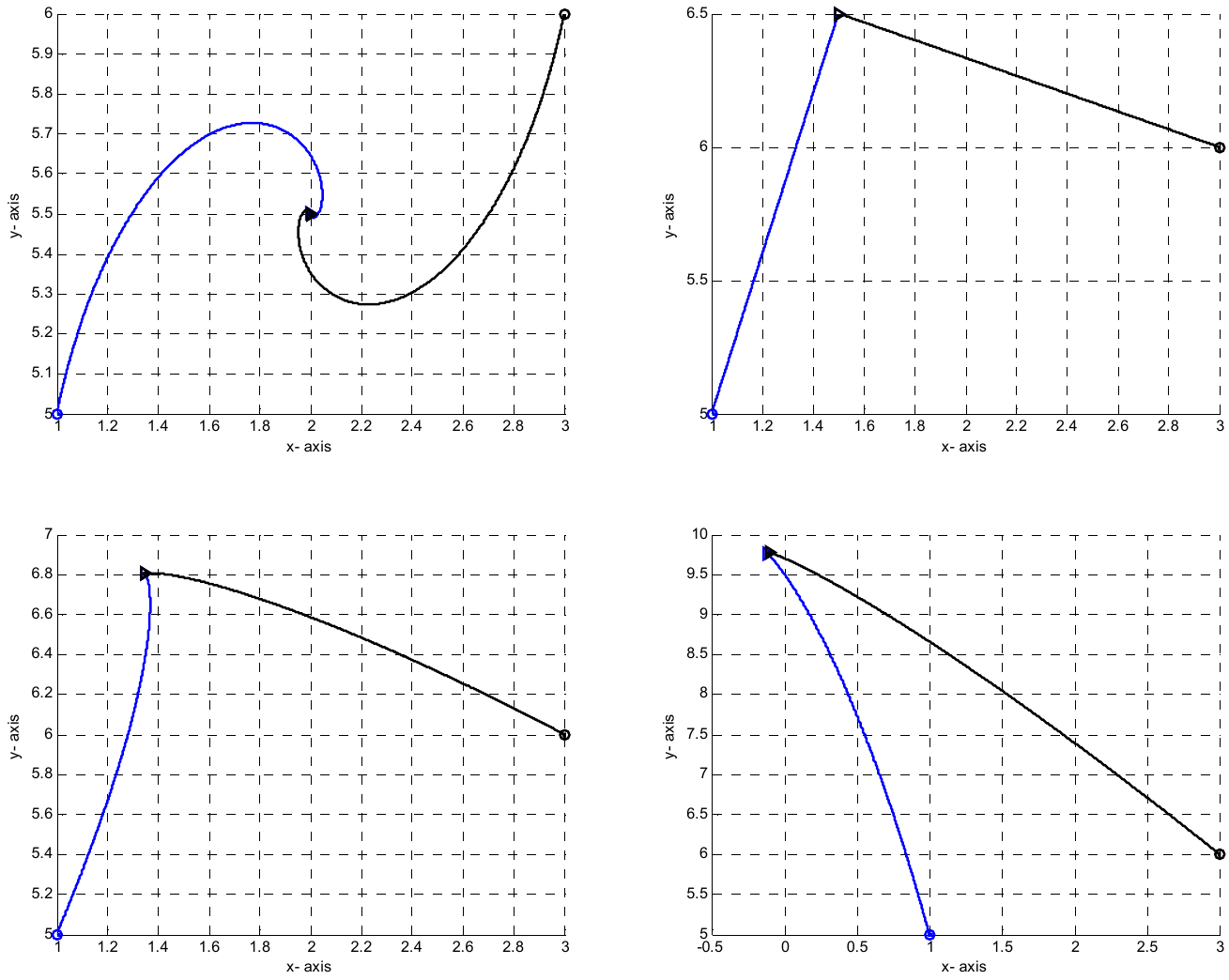}
\caption{Example $1$: Case $2$. Left-top: $\theta_1 = \pi/4$ and $\theta_2 =\pi/4$. Right-top: $\theta_1 = \pi/4$ and $\theta_2 = -\pi/4$. Left-bottom: $\theta_1 = \pi/4$ and $\theta_2 = -\pi/3$. Right-bottom: $\theta_1 = \pi/2.5$ and $\theta_2 = -\pi/2.2$} \label{example1_case2}
\end{figure*}

\subsubsection{Case $3$: With $\pi/2 \leq \theta_1 <\pi$ and $0< \theta_2 \leq \pi/2$} It is the case depicted in Fig.~\ref{example1}(c).
Let us first simulate with $\theta_1 = \pi/2$ and $\theta_2 = -\pi/2$. Since the agents are forced to the same direction, as expected, they are moving in parallel (see the left-top of Fig.~\ref{example1_case3}). With $\theta_1 = \pi/2 + \pi/4$ and $\theta_2 = -\pi/2$, we have the eigenvalues as $0, 0,  -0.707106781186547 \pm j 0.292893218813452$. Since it is unstable, the agents may diverge as shown in the right-top of  Fig.~\ref{example1_case3}. Similarly, with $\theta_1 = \pi/2 + \pi/4$ and $\theta_2 = \pi/2$, we have the eigenvalues as $0, 0,  -0.707106781186547 \pm j 1.707106781186547$. Thus, agents diverge as shown in the left-bottom of Fig.~\ref{example1_case3}. But, with $\theta_1 = \pi/2 +\pi/18$ and $\theta_2 = -\pi/18$, the eigenvalues are $0, 0, 0.811159575345278 \pm j 0.811159575345278$. In this case, the agents converge to a common value; but it is not average consensus as shown in the right-bottom of Fig.~\ref{example1_case3}. 
\begin{figure*}[t]
\centering
\includegraphics[width=12cm]{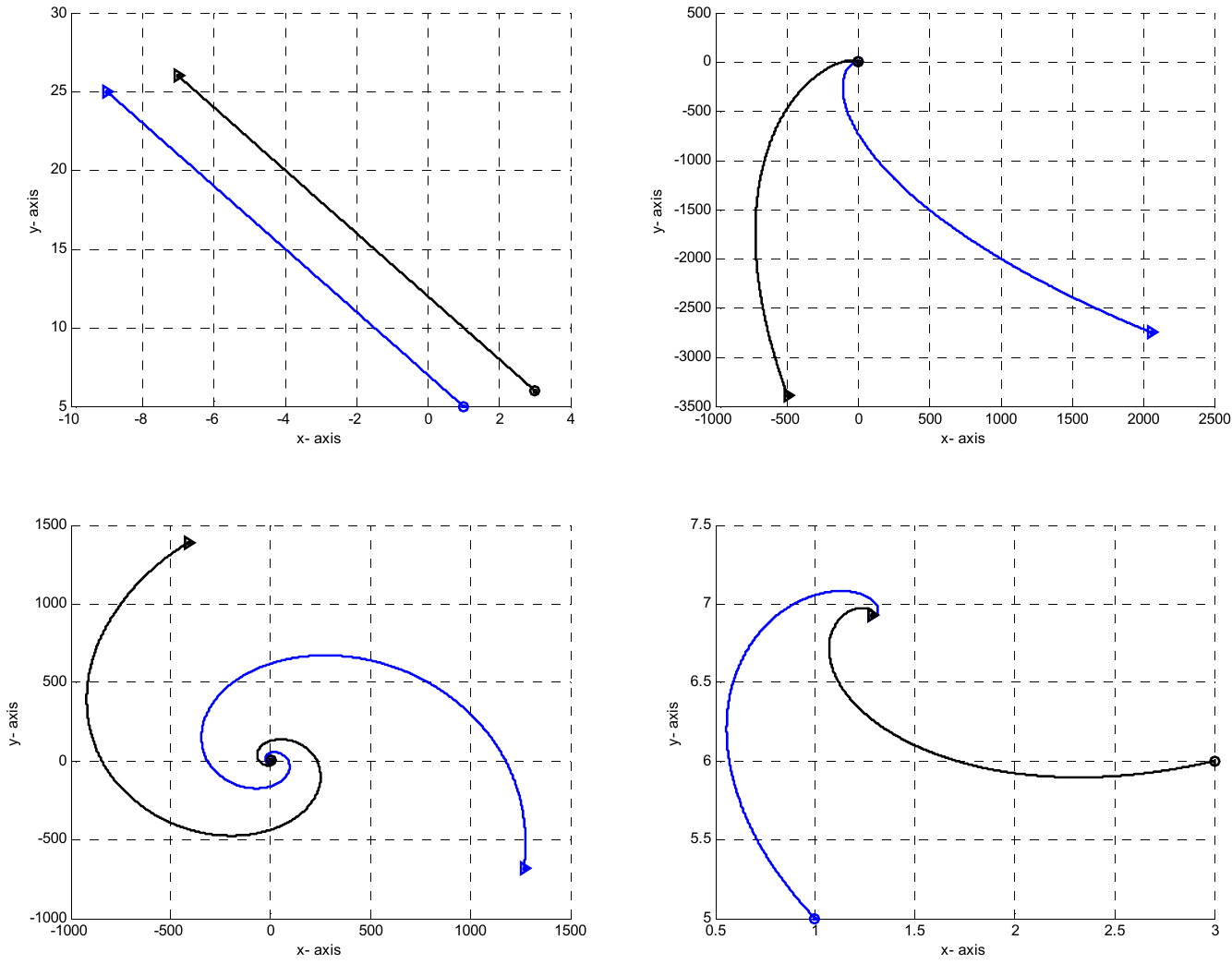}
\caption{Example $1$: Case $3$. Left-top: $\theta_1 = \pi/2$ and $\theta_2 = -\pi/2$. Right-top: $\theta_1 = \pi/2 + \pi/4$ and $\theta_2 = -\pi/2$. Left-bottom: $\theta_1 = \pi/2 + \pi/4$ and $\theta_2 = \pi/2$. Right-bottom: $\theta_1 = \pi/2 +\pi/18$ and $\theta_2 = -\pi/18$} \label{example1_case3}
\end{figure*}

\subsubsection{Case $4$: With $\pi/2< \theta_1 <\pi$ and $\pi/2 < \theta_2 < \pi$} It is the case depicted in Fig.~\ref{example1}(d).
As indicated in Fig.~\ref{example1}(d), agents do not approach each other. So, as expected, an agent is moving far away from other as drawn in Fig.~\ref{example1_case4}. We conduct simulations with $\theta_1 = \pi/2 + \pi/4$ and $\theta_2 = \pi/2 + \pi/4$, which has the eigenvalues $0. 0,  -1.414213562373094 \pm j 1.414213562373095$, and with $\theta_1 = \pi/2 + \pi/4$ and $\theta_2 = -\pi/2 -\pi/4$, which has eigenvalues $0. 0,   -1.414213562373095,  -1.414213562373095$. It is clear that agents do not approach each other, and they are diverging. 
\begin{figure*}[t]
\centering
\includegraphics[width=12cm]{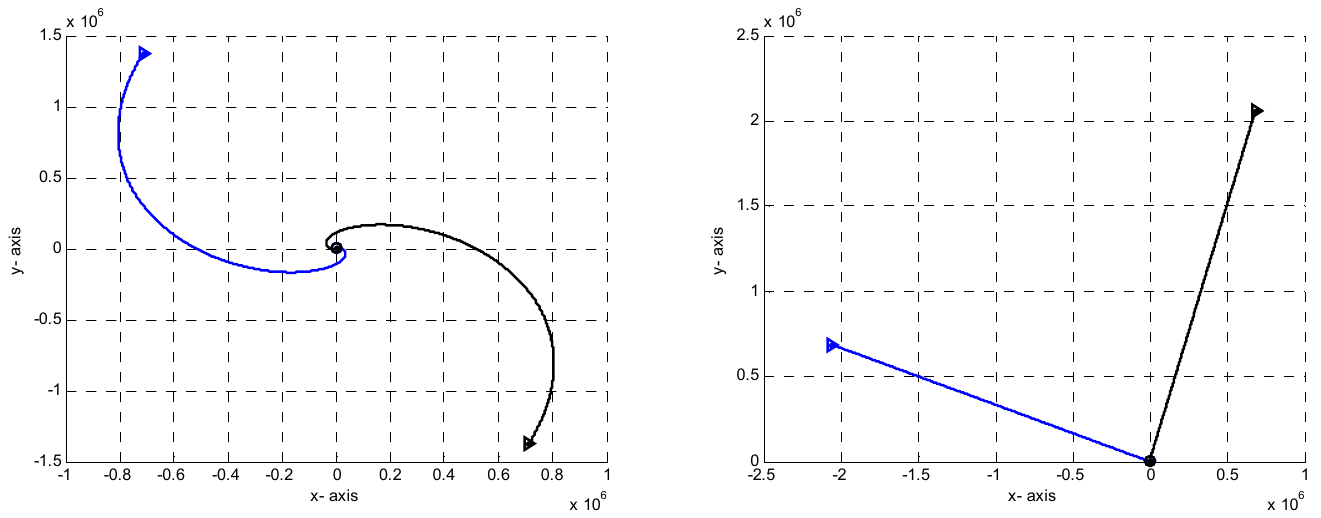}
\caption{Example $1$: Case $4$. Left: $\theta_1 = \pi/2 + \pi/4$ and $\theta_2 = \pi/2 + \pi/4$. Right: $\theta_1 = \pi/2 + \pi/4$ and $\theta_2 = -\pi/2 -\pi/4$} \label{example1_case4}
\end{figure*}

\subsection{Example $2$: Five agents} \label{section_examples_sub2}
As the second example, we consider five agents $1, 2, 3, 4, 5$, with connectivities $a_{1,2}= a_{1,3} = a_{1,4} = a_{1,5} = a_{2,3} = a_{3,5} = a_{4,5}=1$, while with other connectivities of $a_{i,j}=0, \forall i, j,~ i < j$. Table~\ref{example2_eig} shows eigenvalues for the following six cases without two zero eigenvalues.  

\subsubsection{Case 1}
With $\theta_i =0$, it is a traditional consensus, with eigenvalues as shown in Table~\ref{example2_eig}. As shown in the left-top of Fig.~\ref{example2}, the normal average consensus is achieved. 

\subsubsection{Case 2} With $\theta_1=\pi/6$, $\theta_2=-\pi/8$, $\theta_3=\pi/9$, $\theta_4=-\pi/18$, $\theta_5=\pi/25$, all the eigenvalues have positive real parts except two zero eigenvalues. As depicted in the right-top of Fig.~\ref{example2}, a consensus is achieved; but the achieved consensus is not average.

\subsubsection{Case 3}
$\theta_1=\pi/2.1$, $\theta_2=-\pi/2.2$, $\theta_3=\pi/2.1$, $\theta_4=-\pi/2.05$, $\theta_5=-\pi/4$. The eigenvalues have positive real parts; but the magnitude of imaginary values is bigger than the magnitude of real parts. Due to these large imaginary values, the convergence behavior is quite complicated, with a consensus value different from the average value.

\subsubsection{Case 4}
$\theta_1=\pi/6$, $\theta_2=-\pi/2 -\pi/10$, $\theta_3=\pi/9$, $\theta_4=-\pi/18$, $\theta_5=\pi/2 +\pi/10$. There are a pair of eigenvalues located in LHP (see $\lambda_3, \lambda_4$). Due to this unstable eigenvalue, one agent is moving away from the consensus value (see the right-middle of Fig.~\ref{example2}). 

\subsubsection{Case 5}
$\theta_1=\pi/1.8$, $\theta_2=\pi/18$, $\theta_3=0$, $\theta_4=-\pi/18$, $\theta_5=-\pi/8$. The eigenvalues have positive real parts. But, all the eigenvalues have imaginary values. So, the convergence may take a time and the transient behaviors are quite complicated. See the left-bottom of Fig.~\ref{example2}.
  
\subsubsection{Case 6}  
$\theta_1=\pi/2 + \pi/8$, $\theta_2=0$, $\theta_3=0$, $\theta_4=0$, $\theta_5=0$. Although most of eigenvalues are positive real, when an agent that has large connections is misaligned to unstable region, the overall system becomes easily unstable as drawn in right-bottom of Fig.~\ref{example2}.  

\begin{table}
\centering
\caption{Eigenvalues of example $2$}
\label{example2_eig}
\begin{tabular}{ccccc}
\hline
Case & $\lambda_3, \lambda_4$ & $\lambda_5, \lambda_6$ & $\lambda_7, \lambda_8$ &  $\lambda_9, \lambda_{10}$  \\
\hline\hline
Case $1$  &     1.5857, 1.5857        &    3.0, 3.0       &        4.4142, 4.4142       &     4.9999, 4.9999      \\
Case $2$  &    1.6200 $\pm$ j0.3811         &    2.9079 $\pm$ j0.0925       &    4.1838 $\pm$ j0.6911           &    4.3651 $\pm$ j2.0720       \\
Case $3$  &      0.2450 $\pm$ j1.7599        &    0.3534 $\pm$ j4.4676        &     0.7608 $\pm$ j0.7615          &    1.6463 $\pm$ j2.5883       \\
Case $4$  &      -0.0459 $\pm$ j1.4534       &      0.1489 $\pm$ j2.9940      &       2.3166 $\pm$ j0.1147        &  4.2881 $\pm$ j1.9744         \\
Case $5$  &    0.2908 $\pm$ j3.8587          &     1.5750 $\pm$ j0.0556        &        3.2627 $\pm$ j0.2168     &      3.8875 $\pm$ j0.7951      \\
Case $6$  &    -0.5307 $\pm$ j3.6955         &   1.5857, 1.5857        &    3.0, 3.0           &  4.4142, 4.4142         \\ 
\hline
\end{tabular}
\end{table}
\begin{figure*}[!hbt]
\centering
\includegraphics[width=12.5cm]{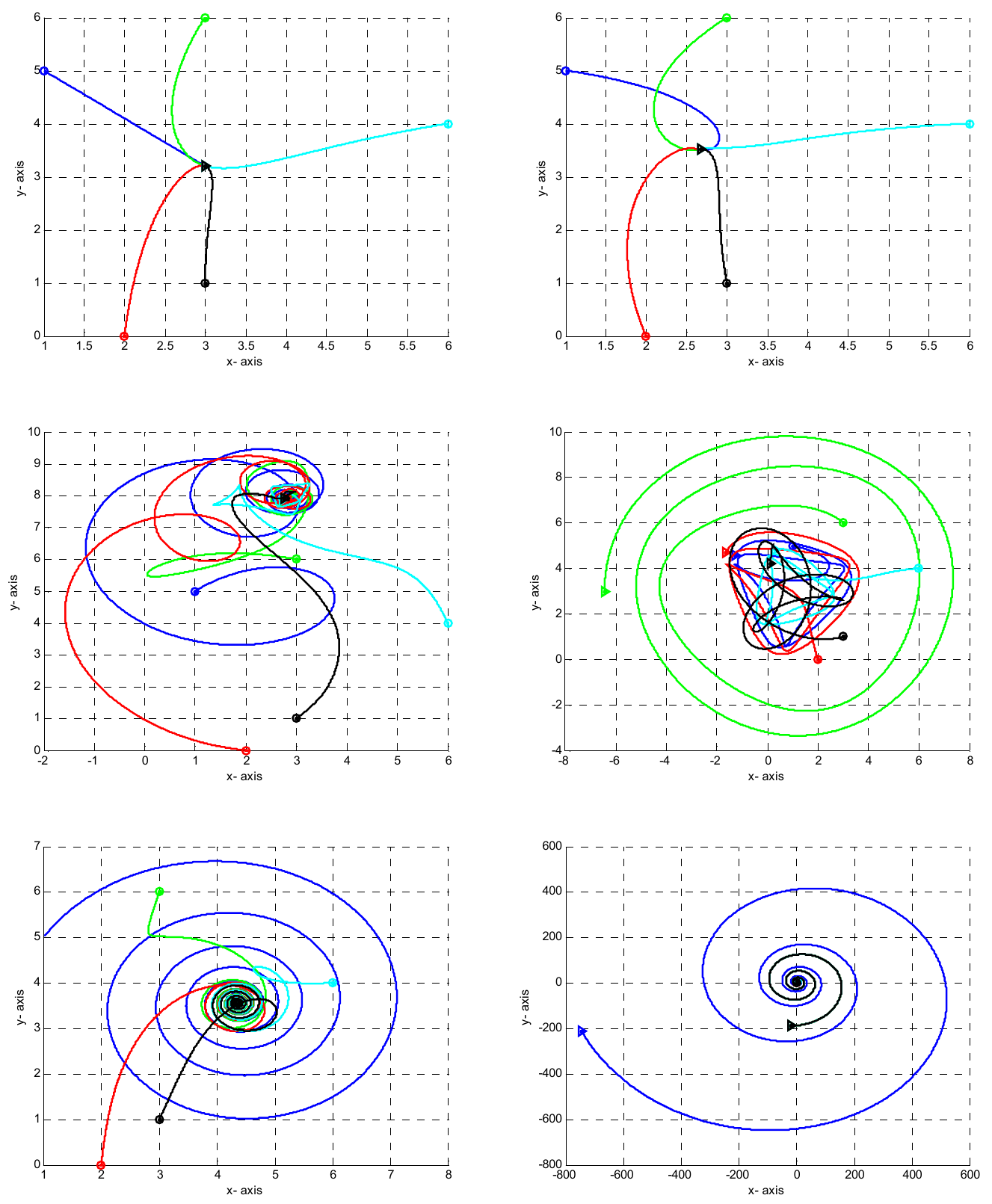}
\caption{Example $2$: Left-top (Case -$1$): $\theta_i =0$. Right-top (Case -$2$): $\theta_1=\pi/6$, $\theta_2=-\pi/8$, $\theta_3=\pi/9$, $\theta_4=-\pi/18$, $\theta_5=\pi/25$. Left-middle (Case -$3$): $\theta_1=\pi/2.1$, $\theta_2=-\pi/2.2$, $\theta_3=\pi/2.1$, $\theta_4=-\pi/2.05$, $\theta_5=-\pi/4$. Right-middle (Case -$4$): $\theta_1=\pi/6$, $\theta_2=-\pi/2 -\pi/10$, $\theta_3=\pi/9$, $\theta_4=-\pi/18$, $\theta_5=\pi/2 +\pi/10$. Left-bottom (Case -$5$): $\theta_1=\pi/1.8$, $\theta_2=\pi/18$, $\theta_3=0$, $\theta_4=-\pi/18$, $\theta_5=-\pi/8$. Right-bottom (Case -$6$): $\theta_1=\pi/2 + \pi/8$, $\theta_2=0$, $\theta_3=0$, $\theta_4=0$, $\theta_5=0$.} \label{example2}
\end{figure*}

\subsection{Example $3$: Three agents under complete graph} \label{section_examples_sub3}
We consider the three agents under complete graph. But for a simplicity, we suppose that agent $1$ has misaligned orientation as $\theta_1$ and agents $2$ and $3$ have the same misalignment angles as $\theta_2 = \theta_3$. Let us first consider $\theta_1 = \pi/1.9$ and $\theta_2 = \theta_3 = -\pi/3$.
As shown in the left plots of Fig.~\ref{example3}, the agents converge to a common point although $\cos\theta_1 <0$. But the converged point is not a average value of initial positions. Next, we slightly change $\theta_1$ as $\theta_1 = \pi/1.6$ with the same $\theta_2 = \theta_3 = -\pi/3$. The right plots of Fig.~\ref{example3} show the trajectories of agents; but in this case, the agents diverge. 

\begin{figure*}[!hbt]
\centering
\includegraphics[width=12cm]{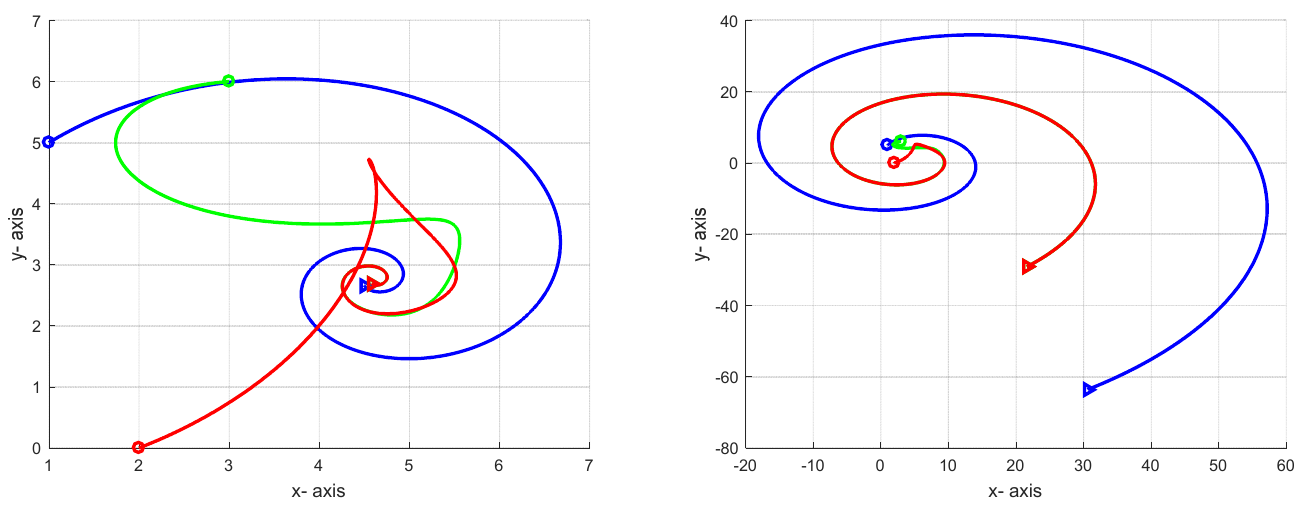}
\caption{Example $3$: Left: $\theta_1 = \pi/1.9$ and $\theta_2 = \theta_3 = -\pi/3$. Right: $\theta_1 = \pi/1.6$ and $\theta_2 = \theta_3 = -\pi/3$} \label{example3}
\end{figure*}

\section{Discussions} \label{section_dis}
As analyzed in Section~\ref{section_analysis}, even with orientation misalignments, the consensus could be achieved if $-\pi/2 < \theta_i <\pi/2$ and the consensus cannot be achieved if $\cos \theta_i < 0,~\forall i \in \mathcal{V}$. Also it was shown that the average consensus is no more ensured when some agents have the orientation misalignments. As remarked in \textit{Remark~\ref{remark_1}}, when some agents have $\theta_i \leq 0$ and some agents have $\theta_i \geq 0$, it is difficult to develop a general result for consensus or non-consensus. With only two agents, as shown in Fig.~\ref{example1_case3}, when $\theta_1 = \pi/2$ and $\theta_2 = -\pi/2$, the agents diverge to infinity. Also with $\theta_1 = 3\pi/4$ and $\theta_2 = -\pi/2$, or  $\theta_1 = 3\pi/4$ and $\theta_2 = \pi/2$, agents diverge. But, with $\theta_1 = \pi/2 +\pi/18$ and $\theta_2 = -\pi/18$, a consensus was achieved.  From these illustrations, it is now observed that when two control input vectors, i.e., $- D(\theta_1) (p_1 - p_2)$ and $- D(\theta_2) (p_2 - p_1)$, are directing opposite directions, the consensus is not achieved (also shown in Fig.~\ref{example1_case4}). But, it is still hard to generalize the results. From the example with five agents, it was also confirmed that the consensus could be achieved if $-\pi/2 < \theta_i <\pi/2$ as illustrated in example $2$-case $2$ and example $2$-case $3$. But, in example $2$-case $4$, when agents $2$ and $5$ have $\theta_2 <0$ and $\theta_5 <0$, respectively, the consensus is not achieved. But, even when again $1$ in example $2$-case $5$ has $\theta_1 <0$, the consensus is achieved, although in case $6$, the consensus was not achieved. Thus, it is clear that the conditions developed in \textit{Theorem~\ref{theorem_three_suff}} are sufficient conditions. 

Obviously the exact condition for a consensus can be achieved by examining eigenvalues of $\mathcal{L}^{out}$ directly. Since the Laplacian matrix $\mathcal{L}^{out}$ is a function of network topology and misaligned orientation angles $\theta_i$, it seems difficult to find the exact locations of eigenvalues (i.e., exact consensus condition) for general graphs. So, for general graphs, the regions defined by \textit{Theorem~\ref{theorem_three_suff}} seem to be best in terms of designing a consensus controller. For example, for two agents, the Laplacian matrix  $\mathcal{L}^{out}$ has eigenvalues as $0, 0, (\cos \theta_1 + \cos \theta_2) \pm j (\sin \theta_1 + \sin \theta_2)$. Thus, a consensus is achieved if and only if $0< \cos \theta_1 + \cos \theta_2$, which is well coincident with the example $1$ of Subsection~\ref{section_examples_sub1}. For three agents under a complete graph in Subsection~\ref{section_examples_sub3}, let us suppose that agent $1$ has misaligned orientation as $\theta_1$ and agents $2$ and $3$ have the same misalignment angles as $\theta_2 = \theta_3$. Then, the eigenvalues of $\mathcal{L}^{out}$ are  $0, 0, 3\cos\theta_2 \pm 3 j \sin\theta_2, 2\cos \theta_1 + \cos\theta_2 \pm j (2 \sin \theta_1 + \sin \theta_2)$. So, if and only if $-\pi/2 < \theta_2 < \pi/2$ and $2\cos \theta_1 + \cos\theta_2>0$, the consensus could be achieved. This analysis is well matched to the simulations given in Subsection~\ref{section_examples_sub3}. However, it is hard to estimate the locations of eigenvalues of general graphs. Hence, it would be recommended to use the sufficient conditions developed in \textit{Theorem~\ref{theorem_three_suff}} for a designing purpose. That is, it is highly recommended to ensure the orientation misalignment errors as $-\pi/2 < \theta_i < \pi/2$.
  


\section{Conclusion}\label{section_conc}
This paper has presented a consensus problem under misalignments of orientations of agents. The three scenarios motivating the misalignment problems, depicted in Fig.~\ref{cood_virtual}, are misaligned orientation errors of local frames, biases in control directions or in sensing directions, and misalignment errors of virtually aligned coordinate frames. We have provided conditions for consensus and added some more analysis related with stability. The locations of eigenvalues have been roughly evaluated by using Gershgorin circle for a block matrix. From the existence of complex eigenvalues locating closely to the imaginary axis, the behaviors of agents are shown to be quite complicated. Also, it was illustrated that the converged consensus point could be quite away from the average of initial positions. Thus, it is highly required to have small misalignment errors; otherwise, the convergence time could be huge and the trajectories of agents become quite complicated. It seems difficult to find exact consensus condition for general graphs. Although we have presented only sufficient conditions, we believe that the conditions could be utilized for a design purpose nicely. In our future efforts, we would be focused on estimating the locations of eigenvalues more tightly in a distributed way. 




\begin{thebibliography}{10}
\providecommand{\url}[1]{#1}
\csname url@samestyle\endcsname
\providecommand{\newblock}{\relax}
\providecommand{\bibinfo}[2]{#2}
\providecommand{\BIBentrySTDinterwordspacing}{\spaceskip=0pt\relax}
\providecommand{\BIBentryALTinterwordstretchfactor}{4}
\providecommand{\BIBentryALTinterwordspacing}{\spaceskip=\fontdimen2\font plus
\BIBentryALTinterwordstretchfactor\fontdimen3\font minus
  \fontdimen4\font\relax}
\providecommand{\BIBforeignlanguage}[2]{{%
\expandafter\ifx\csname l@#1\endcsname\relax
\typeout{** WARNING: IEEEtran.bst: No hyphenation pattern has been}%
\typeout{** loaded for the language `#1'. Using the pattern for}%
\typeout{** the default language instead.}%
\else
\language=\csname l@#1\endcsname
\fi
#2}}
\providecommand{\BIBdecl}{\relax}
\BIBdecl

\bibitem{jadbabaie2003coordination}
A.~Jadbabaie, J.~Lin, and A.~S. Morse, ``Coordination of groups of mobile
  autonomous agents using nearest neighbor rules,'' \emph{IEEE Transactions on
  Automatic Control}, vol.~48, no.~6, pp. 988--1001, 2003.

\bibitem{olfati2007consensuspieee}
J.~A. Olfati-Saber, R.~Fax and R.~M. Murray, ``Consensus and cooperation in
  networked multi-agent systems,'' \emph{Proceedings of the IEEE}, vol.~95,
  no.~1, pp. 215--233, 2007.

\bibitem{Tichakorn_icca_2010}
T.~Wongpiromsarn, K.~You, and L.~Xie, ``A consensus approach to the assignment
  problem: {A}pplication to mobile sensor dispatch,'' in \emph{Proc. of the 8th
  IEEE International Conference on Control and Automation}.\hskip 1em plus
  0.5em minus 0.4em\relax IEEE, June 2010, pp. 2024--2029.

\bibitem{Kim_2015_TCST}
B.-Y. Kim and H.-S. Ahn, ``Consensus-based coordination and control for
  building automation systems,'' \emph{IEEE Trans. on Control Systems
  Technology}, vol.~23, no.~1, pp. 364--371, 2015.

\bibitem{Ziang_tps_2012}
Z.~Zhang and M.-Y. Chow, ``{Convergence Analysis of the Incremental Cost
  Consensus Algorithm Under Different Communication Network Topologies in a
  Smart Grid},'' \emph{IEEE Trans. Power Systems}, vol.~27, no.~4, pp.
  1761--1768, 2012.

\bibitem{weiren_cdc_2008}
W.~Ren, ``Collective motion from consensus with cartesian coordinate coupling -
  {Part I: S}ingle-integrator kinematics,'' in \emph{Proc. of the 47th IEEE
  Conference on Decision and Control}.\hskip 1em plus 0.5em minus 0.4em\relax
  IEEE, Dec. 2008, pp. 1006--1011.

\bibitem{weiren_tac_2009}
------, ``Collective motion from consensus with cartesian coordinate
  coupling,'' \emph{IEEE Trans. on Automatic Control}, vol.~54, no.~6, pp.
  1330--1335, 2009.

\bibitem{Ramirez-Riberos_jgcd_2010}
J.~L. Ramirez-Riberos, M.~Pavone, E.~Frazzoli, and D.~W. Miller, ``Distributed
  control of spacecraft formations via cyclic pursuit: {T}heory and
  experiments,'' \emph{Journal of Guidance, Control, and Dynamics}, vol.~33,
  no.~5, pp. 1655--1669, 2010.

\bibitem{OhAhn_tac_2014}
K.-K. Oh and H.-S. Ahn, ``Formation control and network localization via
  orientation alignment,'' \emph{IEEE Trans. on Automatic Control}, vol.~59,
  no.~2, pp. 540--545, 2014.

\bibitem{LeeAhn_automatica_2016}
B.-H. Lee and H.-S. Ahn, ``Distributed formation control via global orientation
  estimation,'' \emph{Automatica}, vol.~73, no. November, pp. 125--129, 2016.

\bibitem{weiding_automatica_2010}
W.~Ding, G.~Yan, and Z.~Lin, ``Collective motions and formations under pursuit
  strategies on directed acyclic graphs,'' \emph{Automatica}, vol.~46, no.~1,
  pp. 174--181, 2010.

\bibitem{bhlee_icarcv_2016}
B.-H. Lee and H.-S. Ahn, ``Distributed estimation for the unknown orientation
  of the local reference frames in $n$-dimensional,'' in \emph{Proc. of the
  14th International Conference on Control, Automation, Robotics \& Vision},
  Nov., 2016, pp. 1--6.

\bibitem{MinhAhn_arXiv2017}
M.~H. Trinh and H.-S. Ahn, ``Theory and applications of matrix-weighted
  consensus,'' \emph{arXiv:1703.00129v2 [math.OC]}, pp. 1--21, 2017.

\bibitem{Minh_auto_submitted2017}
M.~H. Trinh, C.~V. Nguyen, Y.-H. Lim, and H.-S. Ahn, ``Matrix-weighted
  consensus,'' \emph{submitted for a publication}, pp.~--, 2017.

\bibitem{Feingold_PJM_1962}
D.~G. Feingold and R.~S. Varga, ``Block diagonally dominant matrices and
  generalizations of the {G}ershgorin circle theorem,'' \emph{Pacific Journal
  of Mathematics}, vol.~12, no.~4, pp. 1241--1250, 1962.

\bibitem{Sluis_LAA_1979}
A.~van~der Sluis, ``{G}ershgorln domains for partitioned matrices,''
  \emph{Linear Algebra and Its Applications}, no.~26, pp. 265--280, 1979.

\end{thebibliography}


\end{document}